\documentclass[10pt]{amsart}

\usepackage[latin1]{inputenc}
\usepackage{mathptmx}
\usepackage{amsmath}
\usepackage{dsfont}
\usepackage{amsfonts}
\usepackage{amssymb}
\usepackage{amsthm}
\usepackage{mathrsfs}
\usepackage{MnSymbol}
\usepackage{color,soul}
\usepackage[text={4.5in,7.2in},centering]{geometry}
\usepackage{gensymb}
\usepackage{caption}
\usepackage{float}
\theoremstyle{definition}
 
\numberwithin{dummy}{section}

\theoremstyle{plain}
\newtheorem{thm}{Theorem}[section]

\newtheorem{prop}[thm]{Proposition}

\theoremstyle{definition}
\newtheorem{defn}{Definition}[section]

\theoremstyle{remark}

\DeclareMathAlphabet\mathbb{U}{msb}{m}{n}

\let\oldhat\hat                

\renewcommand{\hat}[1]{\oldhat{\mathbf{#1}}}

\usepackage{calc}
\usepackage{tikz}

\usetikzlibrary{decorations.markings}
\usetikzlibrary{decorations.pathreplacing}
\usetikzlibrary{positioning,automata}
\usetikzlibrary{shapes}
\usetikzlibrary{arrows}
\usepackage{verbatim}
\tikzset{
	>=stealth',
	vertex/.style={
		circle,
		draw=#1,
		fill=#1,
		inner sep = 2pt,
		outer sep = 0pt,
		text centered},
	edge/.style={
		-,
		thin,
		draw=black,
		},
	text style/.style={
		sloped,
		text=black,
		font=\normalsize,
		above}
	}

\tikzset{every loop/.style={min distance=10mm,in=45,out=135,looseness=0}}

\title{Perfect Domination in Knights Graphs }

\author{Todd Fenstermacher, Soumendra Ganguly, Renu Laskar \\ Clemson University}

\begin{document}
	
	\begin{abstract}  
		For a graph $G = (V,E),$ a subset $S$ of $V$ is a \textit{perfect dominating set} of $G$ if every vertex not in $S$ is adjacent to exactly one vertex in $S.$ The perfect domination number, $\gamma_p(G),$ is the minimum cardinality of a perfect dominating set of $G.$  The perfect domination number is found for knights graphs on square, rectangular, and infinite chessboards. Indeed, exact values or bounds are given for all chessboards except those with 3 rows and number of columns congruent to 1, 2, or 3 modulo 8. 
	\end{abstract}
	
	\maketitle
	
	
\section {Introduction}

For a graph $G = (V,E),$ we define a perfect dominating set as in \cite{Cockayne}. 

\begin{defn}
	For a graph $G = (V,E),$ a set $S \subseteq V$ is a \textit{perfect dominating set} if $\forall v \in V - S,$ $|N(v) \cap S| = 1.$ 		
\end{defn}
Here $N(v)$ is the open neighborhood of $v$ and contains all vertices adjacent to $v.$ The \textit{perfect domination number} is the smallest such dominating set, that is 
\[\gamma_p(G) = \min\{ |S| : S \text{ is a perfect dominating set of }G	\}.\]

A concept similar to perfect domination is efficient domination, which we define next. 
	
\begin{defn}
	For a graph $G = (V,E),$ a set $S \subseteq V$ is an \textit{efficient dominating set} if $\forall v \in V,$ $|N[v] \cap S| = 1.$ 
\end{defn}

Here $N[v]$ is the closed neighborhood of $v$ and contains all vertices adjacent to $v,$ as well as $v$ itself. Now, following \cite{Sinko}, we define the knights graph $KN_{n,m}.$ 
\begin{defn} 
	The knights graph $KN_{n,m},$ is a graph of order $nm$ where each vertex represents a square on a chessboard with $n$ columns and $m$ rows, and two vertices are adjacent if a knight can move between the two squares corresponding to these two vertices. 
\end{defn}

Two representations of $KN_{4,4}$ are given in Figure \ref{4x4}. In general, the representation of $KN_{n,m}$ using vertices and edges is visually convoluted, so the representation of $KN_{n,m}$ using a board is preferred. When referring to specific squares of the chessboard, we use ordered pairs $(i,j)$ with $1 \leq i \leq n$ and $1 \leq j \leq m.$ Note that we follow the chessboard standard of naming the column first and row second (see \cite{Elkies},\cite{Sinko}) so that $(i,j)$ represents the square in column $i$ and row $j.$ Again, note this is the opposite of the matrix notation where $(i,j)$ would represents row $i$ and column $j.$ Unless context specifies otherwise, the bottom left square is assumed to be $(1,1),$ and square $(i,j)$ represents the square in column $i$ and row $j.$ Furthermore, in $KN_{n,m},$ the squares adjacent to square $(a,b)$ are $\{ (a\pm1,b\pm2)\}\bigcup \{(a\pm2,b\pm1)\}.$ For example, the closed neighborhood of $(3,4)$ in $KN_{8,8}$ is 
\[ N[(3,4)] = \{  (1,3),(1,5),(2,2),(2,6),(3,4),(4,2),(4,6),(5,3),(5,5)  \}. \]
Sinko and Slater studied efficient domination on knights graphs in \cite{Sinko}. Here we study perfect domination on knights graphs. We will first explore finite knights graphs before turning to infinite knights graphs. In our investigation of finite knights graphs we make extensive use of computer searches. All code used is available upon request from the first author.  

		\begin{figure}
			\begin{tikzpicture}
			\foreach \x in {0,...,3}
			\foreach \y in {0,...,3} { 
				\draw node at (\x, \y) [vertex,draw] {}; }

			\draw[-] (0,0) -- (1,2);
			\draw[-] (0,0) -- (2,1);
			\draw[-] (1,0) -- (0,2);
			\draw[-] (1,0) -- (2,2);
			\draw[-] (1,0) -- (3,1);
			\draw[-] (2,0) -- (1,2);
			\draw[-] (2,0) -- (3,2);
			\draw[-] (2,0) -- (0,1);
			\draw[-] (3,0) -- (1,1);
			\draw[-] (3,0) -- (2,2);
			\draw[-] (0,1) -- (1,3);
			\draw[-] (0,1) -- (2,2);
			\draw[-] (1,1) -- (0,3);
			\draw[-] (1,1) -- (2,3);
			\draw[-] (1,1) -- (3,2);
			\draw[-] (2,1) -- (0,2);
			\draw[-] (2,1) -- (1,3);
			\draw[-] (2,1) -- (3,3);
			\draw[-] (3,1) -- (1,2);
			\draw[-] (3,1) -- (2,3);
			\draw[-] (0,2) -- (2,3);
			\draw[-] (1,2) -- (3,3);
			\draw[-] (2,2) -- (0,3);
			\draw[-] (3,2) -- (1,3);
			

			\foreach \i in {0,...,4} {
				\draw [very thick, black] (.75*\i + 5,0) -- (.75*\i+5,3);
			}
			\foreach \i in {0,...,4} {
				\draw [very thick,black] (5,.75*\i) -- (8,.75*\i);
			}
			
			
			\end{tikzpicture}
			\caption{Two representations of $KN_{4,4}$}
			\label{4x4}
		\end{figure}
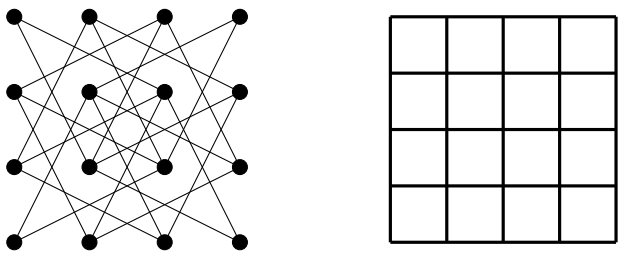

\section {Finite Boards}

To place our question on an actual chessboard would, of course, mandate using $KN_{8,8}.$ However, the only boards with nontrivial perfect dominating knights have fewer than $5$ rows or $5$ columns. This is a result of the following theorem. 

\begin{thm}\label{5x5}
	 $\gamma_p(KN_{n,m}) = nm$ for $m,n \geq 5.$ 
\end{thm}

\begin{proof}
	Let $m,n \geq 5$ and consider $KN_{n,m}.$ First construct a nontrivial set $S$ of knights which perfectly dominates the lower left $3\times 3$ sub-board. That is, construct $S$ by adding only knights which are necessary to perfectly dominate the $3 \times 3$ sub-board. 
	
	Accounting for symmetry, there are 13 distinct ways to construct $S.$ However, for each such construction, there are squares in the lower left $5 \times 5$ sub-board which are not dominated. Moreover, adding knights to any of these constructions in order to perfectly dominate these squares necessitates placing a knight on every square of the board. To demonstrate this, we divide the 13 different constructions of $S$ into 3 cases. Note that in the following cases we only include squares in $S$ which are also in the lower left $5 \times 5$ sub-board so that each construction makes sense in the case $n=m=5.$ \\
	
	Case 1. 
	\begin{eqnarray*}
		& & S = \{ (1,3),(2,3),(3,1),(4,1)  \} \text{ (see Figure \ref{S1})}	\\
\text{or} & & S = \{ (1,1),(1,3),(2,1),(2,4),(3,2),(4,3),(5,1)  \} \\
\text{or} & & S = \{ (1,2),(1,3),(2,4),(3,1),(3,2),(4,3),(5,1)  \} \\
\text{or} & & S = \{ (1,2),(1,5),(2,3),(3,1),(3,4),(4,2),(5,3) \}
	\end{eqnarray*}

		\begin{figure}[H]
		\begin{tikzpicture}
		\foreach \i in {0,...,5} {
			\draw [very thick,black] (.5*\i-.25,-.25) -- (.5*\i-.25,2.75);
		}
		\foreach \i in {0,...,5} {
			\draw [very thick,black] (-.25,.5*\i-.25) -- (2.75,.5*\i-.25);
		}

		\draw node at (0,1) {N};
		\draw node at (.5,1) {N};
		\draw node at (1,0) {N};
		\draw node at (1.5,0) {N};
		
		\draw node at (0,0) {X};
		\draw node at (0,.5) {X};
		\draw node at (0,2) {X};
		\draw node at (.5,0) {X};
		\draw node at (.5,.5) {X};
		\draw node at (.5,2) {X};
		\draw node at (1,.5) {X};
		\draw node at (1,1) {X};
		\draw node at (1,1.5) {X};
		\draw node at (1,2) {X};
		\draw node at (1.5,.5) {X};
		\draw node at (1.5,1) {X};
		\draw node at (1.5,1.5) {X};
		\draw node at (2,.5) {X};
		\draw node at (2,1) {X};
		\draw node at (2.5,.5) {X};

		\end{tikzpicture}
		\caption{ N represents a knight, X represents a dominated square }
		\label{S1}
		\end{figure}
	
For each of these four possibilities of $S,$ the square $(1,4)$ is not dominated. Moreover, adding any knight to $S$ to dominate $(1,4)$ forces knights to be placed on every square of $KN_{n,m}.$ \\

	Case 2. 
	\begin{eqnarray*}
		  & & S = \{(1,1),(2,5),(3,3),(4,3)\} \\
\text{or} & & S = \{(1,1),(2,1),(3,3),(4,3)\} \\
\text{or} & & S = \{(1,1),(1,2),(1,4),(2,5),(3,3)\} \\
\text{or} & & S = \{(1,1),(1,2),(2,5),(3,3),(4,1)\} \\
\text{or} & & S = \{(1,1),(1,2),(1,4),(2,1),(3,3)\} \\
\text{or} & & S = \{(1,3),(2,4),(3,2),(3,5),(4,3),(5,1),(5,4)\} \\
\text{or} & & S = \{(1,3),(2,4),(2,5),(3,2),(4,3),(4,4),(5,1)\} \\
\text{or} & & S = \{(1,3),(1,4),(2,1),(2,5),(3,2),(3,3),(4,4),(5,2)\} \\ 	
	\end{eqnarray*}
	
	For any choice of $S$ above, the square $(1,5)$ cannot be dominated without being forced to place a knight in every square in $KN_{n,m}.$ \\
	
	Case 3. 
	\[ S = \{ (1,1),(1,2),(2,4),(3,2),(3,3),(4,1),(4,5),(5,3)\} \]
	
	For $ m \leq 6,$ the square $(1,5)$ cannot be dominated without placing a knight in every square of $KN_{n,m}.$ However, for $m \geq 7,$ $S$ may be extended to dominate $(1,5)$ without placing a knight on every square. This is only possible by placing a knight on $(2,7).$ However, $S\bigcup\{(2,7)\}$ does not dominate $(5,5),$ and adding a knight to dominate $(5,5)$ necessitates a knight be placed on every square of $KN_{n,m}.$ 
	
	Therefore, there is no nontrivial perfect dominating set of $KN_{n,m}$ for $n,m \geq 5,$ that is, $\gamma_p(KN_{n,m}) = nm.$	
\end{proof}

Now of course $\gamma_p(KN_{n,1}) = n$ for all $n.$ It is also simple to show that $\gamma_p(KN_{6k+1,2}) = 4k+2$ for $k = 0,1,2,\dots$ and $\gamma_p(KN_{n,2}) = 4k$ for $6k - 4 \leq n \leq 6k$ and $k = 1,2,3,\dots.$ We record this fact as a proposition below. 

\begin{prop}
	$\gamma_p(KN_{6k+1,2}) = 4k + 2$ for $k = 0,1,2,\dots.$ 
	
	$\gamma_p(KN_{n,2}) = 4k$ for $6k - 4 \leq n \leq 6k$ and $k=1,2,3,\dots.$ 
\end{prop}

%
%
%
%

As a result, there are two interesting cases when dealing with finite boards - boards with 3 or 4 rows. We first look at boards with 3 rows. We give bounds for boards with the number of columns congruent to $0,4,5,6,$ or $7$ modulo 8. 

\begin{prop}\label{3xn}
	$\gamma_p(KN_{8k,3}) \leq 10k$ for $k = 1,2,3,\dots$
	
	$\gamma_p(KN_{8k+4,3}) \leq 10k + 6$ for $k = 0,1,2,\dots$
	
	$\gamma_p(KN_{8k+5,3}) \leq 10k + 6$ for $k = 0,1,2,\dots$
	
	$\gamma_p(KN_{8k+6,3}) \leq 10k + 7$ for $k = 0,1,2,\dots$
	
	$\gamma_p(KN_{8k+7,3}) \leq 10k + 9$ for $k = 0,1,2,\dots$	
\end{prop}

\begin{proof}
	
	For each of the bounds in Theorem \ref{3xn} we present the a placement of knights that can be extended indefinitely. 
	
	$\gamma_p(KN_{8k,3}) \leq 10k$ 
	
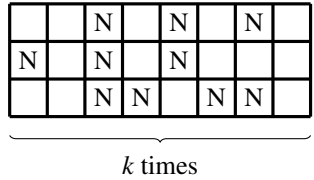
	\begin{center}
	\begin{tikzpicture}
	\foreach \i in {0,...,8} {
		\draw [very thick,black] (.5*\i-.25,-.25) -- (.5*\i-.25,1.25);
	}
	\foreach \i in {0,...,3} {
		\draw [very thick,black] (-.25,.5*\i-.25) -- (3.75,.5*\i-.25);
	}

	\draw node at (0,.5) {N};
	\draw node at (1,0) {N};
	\draw node at (1,.5) {N};
	\draw node at (1,1) {N};
	\draw node at (1.5,0) {N};
	\draw node at (2,.5) {N};
	\draw node at (2,1) {N};
	\draw node at (2.5,0) {N};
	\draw node at (3,0) {N};
	\draw node at (3,1) {N};
	
	\draw [decorate, decoration = {brace, mirror}] (-.25,-.5) -- (3.75,-.5) node [midway,yshift = -.4cm] {$k$ times};
	
	\end{tikzpicture}
	\end{center}

$\gamma_p(KN_{8k+4,3}) \leq 10k + 6$ 

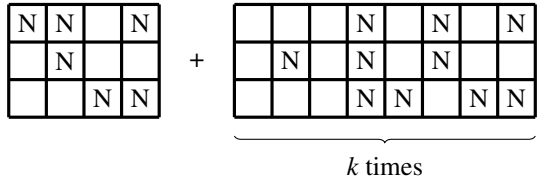
\begin{center}
	\begin{tikzpicture}
	\foreach \i in {0,...,4} {
		\draw [very thick, black] (.5*\i+1.25,-.25) -- (.5*\i+1.25,1.25);
	}
	\foreach \i in {0,...,3} {
		\draw [very thick,black] (1.25,.5*\i-.25) -- (3.25,.5*\i-.25);
	}
	
	\draw node at (1.5,1) {N};
	\draw node at (2,.5) {N};
	\draw node at (2,1) {N};
	\draw node at (2.5,0) {N};
	\draw node at (3,0) {N};
	\draw node at (3,1) {N};

	\draw node at (3.75,.5) {+};

	\foreach \i in {0,...,8} {
		\draw [very thick,black] (.5*\i+4.25,-.25) --  (.5*\i+4.25,1.25);
	}
	\foreach \i in {0,...,3} {
		\draw [very thick,black] (4.25,.5*\i-.25) -- (8.25,.5*\i-.25);
	}

	\draw node at (5,.5) {N};
	\draw node at (6,0) {N};
	\draw node at (6,.5) {N};
	\draw node at (6,1) {N};
	\draw node at (6.5,0) {N};
	\draw node at (7,.5) {N};
	\draw node at (7,1) {N};
	\draw node at (7.5,0) {N};
	\draw node at (8,0) {N};
	\draw node at (8,1) {N};
	
	\draw [decorate, decoration = {brace, mirror}] (4.25,-.5) -- (8.25,-.5) node [midway,yshift = -.4cm] {$k$ times};
	
	\end{tikzpicture}
\end{center}

$\gamma_p(KN_{8k+5,3}) \leq 10k + 6$

\begin{center}
	\begin{tikzpicture}
	\foreach \i in {0,...,5} {
		\draw [very thick, black] (.5*\i+.75,-.25) -- (.5*\i+.75,1.25);
	}
	\foreach \i in {0,...,3} {
		\draw [very thick,black] (.75,.5*\i-.25) -- (3.25,.5*\i-.25);
	}
	
	\draw node at (1,1) {N};
	\draw node at (1.5,.5) {N};
	\draw node at (1.5,1) {N};
	\draw node at (2,0) {N};
	\draw node at (2.5,0) {N};
	\draw node at (2.5,1) {N};

	\draw node at (3.75,.5) {+};

	\foreach \i in {0,...,8} {
		\draw [very thick,black] (.5*\i+4.25,-.25) --  (.5*\i+4.25,1.25);
	}
	\foreach \i in {0,...,3} {
		\draw [very thick,black] (4.25,.5*\i-.25) -- (8.25,.5*\i-.25);
	}

	\draw node at (4.5,.5) {N};
	\draw node at (5.5,0) {N};
	\draw node at (5.5,.5) {N};
	\draw node at (5.5,1) {N};
	\draw node at (6,0) {N};
	\draw node at (6.5,.5) {N};
	\draw node at (6.5,1) {N};
	\draw node at (7,0) {N};
	\draw node at (7.5,0) {N};
	\draw node at (7.5,1) {N};
	
	\draw [decorate, decoration = {brace, mirror}] (4.25,-.5) -- (8.25,-.5) node [midway,yshift = -.4cm] {$k$ times};
	
	\end{tikzpicture}
\end{center}

	$\gamma_p(KN_{8k+6,3}) \leq 10k + 7$
	
	
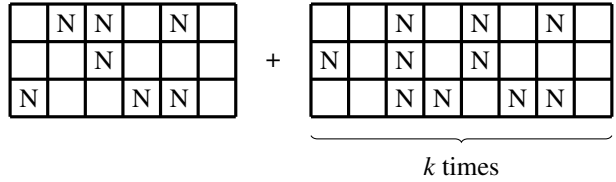
\begin{center}
		\begin{tikzpicture}
		\foreach \i in {0,...,6} {
			\draw [very thick, black] (.5*\i+.25,-.25) -- (.5*\i+.25,1.25);
		}
		\foreach \i in {0,...,3} {
			\draw [very thick,black] (.25,.5*\i-.25) -- (3.25,.5*\i-.25);
		}
		
		\draw node at (.5,0) {N};
		\draw node at (1,1) {N};
		\draw node at (1.5,.5) {N};
		\draw node at (1.5,1) {N};
		\draw node at (2,0) {N};
		\draw node at (2.5,0) {N};
		\draw node at (2.5,1) {N};

		\draw node at (3.75,.5) {+};

		\foreach \i in {0,...,8} {
			\draw [very thick,black] (.5*\i+4.25,-.25) --  (.5*\i+4.25,1.25);
		}
		\foreach \i in {0,...,3} {
			\draw [very thick,black] (4.25,.5*\i-.25) -- (8.25,.5*\i-.25);
		}

		\draw node at (4.5,.5) {N};
		\draw node at (5.5,0) {N};
		\draw node at (5.5,.5) {N};
		\draw node at (5.5,1) {N};
		\draw node at (6,0) {N};
		\draw node at (6.5,.5) {N};
		\draw node at (6.5,1) {N};
		\draw node at (7,0) {N};
		\draw node at (7.5,0) {N};
		\draw node at (7.5,1) {N};
		
		\draw [decorate, decoration = {brace, mirror}] (4.25,-.5) -- (8.25,-.5) node [midway,yshift = -.4cm] {$k$ times};
		
		\end{tikzpicture}
	\end{center}

	$\gamma_p(KN_{8k+7,3}) \leq 10k + 9$
	
		
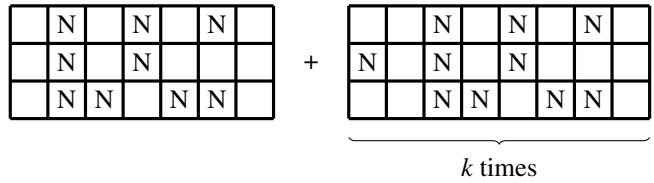
\begin{center}
			\begin{tikzpicture}
			\foreach \i in {0,...,7} {
				\draw [very thick, black] (.5*\i-.25,-.25) -- (.5*\i-.25,1.25);
			}
			\foreach \i in {0,...,3} {
				\draw [very thick,black] (-.25,.5*\i-.25) -- (3.25,.5*\i-.25);
			}
			
			\draw node at (.5,0) {N};
			\draw node at (.5,.5) {N};
			\draw node at (.5,1) {N};
			\draw node at (1,0) {N};
			\draw node at (1.5,.5) {N};
			\draw node at (1.5,1) {N};
			\draw node at (2,0) {N};
			\draw node at (2.5,0) {N};
			\draw node at (2.5,1) {N};

			\draw node at (3.75,.5) {+};

			\foreach \i in {0,...,8} {
				\draw [very thick,black] (.5*\i+4.25,-.25) --  (.5*\i+4.25,1.25);
			}
			\foreach \i in {0,...,3} {
				\draw [very thick,black] (4.25,.5*\i-.25) -- (8.25,.5*\i-.25);
			}

			\draw node at (4.5,.5) {N};
			\draw node at (5.5,0) {N};
			\draw node at (5.5,.5) {N};
			\draw node at (5.5,1) {N};
			\draw node at (6,0) {N};
			\draw node at (6.5,.5) {N};
			\draw node at (6.5,1) {N};
			\draw node at (7,0) {N};
			\draw node at (7.5,0) {N};
			\draw node at (7.5,1) {N};
			
			\draw [decorate, decoration = {brace, mirror}] (4.25,-.5) -- (8.25,-.5) node [midway,yshift = -.4cm] {$k$ times};
			
			\end{tikzpicture}
		\end{center}
	
\end{proof}

Now we turn our attention to boards with 4 rows. We will show that for boards with a large number of columns, there is exactly one nontrivial perfect dominating set, and such a set exists only for boards with an even number of columns. 

\begin{prop}\label{4_rows}
	$\gamma_p(KN_{2k,4}) = 4k$ for $k \geq 7.$ 
	$\gamma_p(KN_{2k+1,4}) = 4(2k+1)$ for $k \geq 6.$ 	
\end{prop}

\begin{proof}
	We again make use of a computer in this proof. Consider a set $S$ of knights which perfectly dominates the middle $n-4$ columns of $KN_{n,4}$ for arbitrary $n\geq 5.$ Starting at $n=5,$ we find all such possible sets $S$ which perfectly dominate the middle column of $KN_{5,4}.$ Then, we find all possible extensions of these sets to dominate the middle 2 columns of $KN_{6,4}.$ In this manner, we successively increase $n$ by $1$ and consider how to dominate the middle $n-4$ columns, using the possible sets of knights found in the previous step. When considering $KN_{13,4}$ there is exactly one set of knights (up to symmetry) which perfectly dominates the middle 9 columns (see Figure \ref{4xn} for a demonstration of this patter on $KN_{14,4}$). Extending this set of knights to a perfect dominating set of $KN_{13,4}$ forces a knight to be placed on every square. Moreover, if this placement of knights is put on any larger board (i.e., $KN_{n,4}$ with $n \geq 14$) in an attempt to find a perfect dominating set, all additional knights are forced according to this pattern. However, this pattern of knights only perfectly dominates $KN_{n,4}$ for even $n.$ Thus, there exists no nontrivial perfect dominating set of $KN_{n,4}$ for odd $n \geq 13.$ 
\end{proof}
	
		\begin{center}
			\begin{figure}[H]
			\begin{tikzpicture}
			\foreach \i in {0,...,14} {
				\draw [very thick,black] (.5*\i-.25,-.25) -- (.5*\i-.25,1.75);
			}
			\foreach \i in {0,...,4} {
				\draw [very thick,black] (-.25,.5*\i-.25) -- (6.75,.5*\i-.25);
			}

			\draw node at (0,0) {N};
			\draw node at (0,1) {N};
			\draw node at (.5,.5) {N};
			\draw node at (.5,1.5) {N};
			\draw node at (1,.5) {N};
			\draw node at (1,1.5) {N};
			\draw node at (1.5,0) {N};
			\draw node at (1.5,1) {N};
			\draw node at (2,0) {N};
			\draw node at (2,1) {N};
			\draw node at (2.5,.5) {N};
			\draw node at (2.5,1.5) {N};
			\draw node at (3,.5) {N};
			\draw node at (3,1.5) {N};
			\draw node at (3.5,0) {N};
			\draw node at (3.5,1) {N};
			\draw node at (4,0) {N};
			\draw node at (4,1) {N};
			\draw node at (4.5,.5) {N};
			\draw node at (4.5,1.5) {N};
			\draw node at (5,.5) {N};
			\draw node at (5,1.5) {N};
			\draw node at (5.5,0) {N};
			\draw node at (5.5,1) {N};
			\draw node at (6,0) {N};
			\draw node at (6,1) {N};			
			\draw node at (6.5,.5) {N};
			\draw node at (6.5,1.5) {N};			
						
			\end{tikzpicture}
			\caption{Perfect Dominating set of $KN_{14,4}$}
			\label{4xn}			
			\end{figure}
		\end{center}

The following are results quickly computed using the algorithm described in the proof of Proposition \ref{4_rows}, and are given without proof. 
\begin{prop}
	$\gamma_p(KN_{4,4}) = \gamma_p(KN_{5,4}) = \gamma_p(KN_{6,4}) = 8$
	
	$\gamma_p(KN_{7,4}) = 28$
	
	$\gamma_p(KN_{8,4}) = 16 $
	
	$\gamma_p(KN_{9,4}) = 36$
	
	$\gamma_p(KN_{10,4}) = \gamma_p(KN_{11,4}) = \gamma_p(KN_{12,4}) = 16$	
\end{prop}

\section{Infinite Boards}

We now turn our attention to infinite chessboards. Following \cite{Sinko}, we look at five different infinite chessboards. The notation used and a brief description of each infinite board is given below.  
\begin{itemize}
	\item $KN_{\mathbb{Z},\mathbb{Z}}$: chessboard without boundaries 	
	\item $KN_{\mathbb{Z},m}$: finitely many rows, infinitely many  columns 
	\item $KN_{\mathbb{N},m}$: finitely many rows, infinitely many columns, opening to the right 
	\item $KN_{\mathbb{N},\mathbb{N}}$: infinitely many rows, infinitely many columns, one corner 
	\item $KN_{\mathbb{Z},\mathbb{N}}$: infinitely many rows, infinitely man columns, one horizontal boundary 		
\end{itemize}

For each of the infinite chessboards we ask the following two questions: ``Does a nontrivial perfect dominating set exist?" and, if so, ``What is minimum proportion of the squares containing knights in such a set?" The second questions is analogous to finding the nontrivial perfect dominating set of minimum cardinality on a finite chessboard. For convenience, we refer to the proportion of squares containing knights as the \textit{density} of the set of knights. We can, in fact, answer both questions for nearly each infinite board. We start with  $KN_{\mathbb{Z},\mathbb{Z}}.$  

\begin{prop}
	There exists a nontrivial perfect dominating set of $KN_{\mathbb{Z},\mathbb{Z}}$ with knights populating $1/8$ of the board. Moreover, $1/8$ is the minimum density of a perfect dominating set.
\end{prop}

\begin{proof}
	Figure \ref{infinite_1} gives an example of a perfect dominating set with density $1/8.$ 	We must now show that $1/8$ is the minimum density of a perfect dominating set of $KN_{\mathbb{Z},\mathbb{Z}}.$ We do so by showing that no such set can contain an isolated knight. So suppose we have  a nontrivial perfect dominating set, and for sake of contradiction, assume there exists an isolated knight.  WLOG call this knight $N_1$ and set it in square (0,0) to be the origin. 

	\begin{figure}[H]
		\begin{tikzpicture}
		\foreach \i in {0,...,10} {
			\draw [very thick,black] (.5*\i-2.25,-2.75) -- (.5*\i-2.25,2.75);
		}
		\foreach \i in {0,...,9} {
			\draw [very thick,black] (-2.75,.5*\i-2.25) -- (3.25,.5*\i-2.25);
		}
		
		\draw node at (0,0) {N};	
		\draw node at (.5,0) {N};	
		\draw node at (1,-1) {N};	
		\draw node at (1.5,-1) {N};		
		\draw node at (-.5,1) {N};	
		\draw node at (-1,1) {N};		
		\draw node at (-1.5,2) {N};	
		\draw node at (-2,2) {N};	
		\draw node at (-1,-2.5) {N};		
		\draw node at (-1.5,-2.5) {N};
		\draw node at (-2,-1.5) {N};
		\draw node at (-2.5,-1.5) {N};	
		\draw node at (1.5,2.5) {N};		
		\draw node at (2,2.5) {N};		
		\draw node at (2.5,1.5) {N};		
		\draw node at (3,1.5) {N};			
		\draw node at (2,-2) {N};
		\draw node at (2.5,-2) {N};
		\end{tikzpicture}
		\caption{A perfect dominating set of $KN_{\mathbb{Z},\mathbb{Z}}$ with density $1/8.$}
		\label{infinite_1}
	\end{figure}
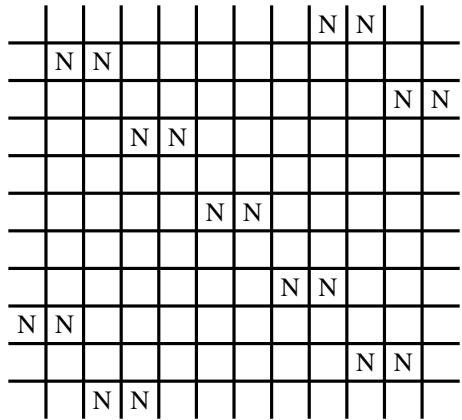

	Because $N_1$ must be isolated our board locally takes the form shown in Figure \ref{N_1}, where squares with an $X$ may not contain a knight. Note that if any square $(a,b)$ with an X in Figure \ref{N_1} is dominated, then no other square in the closed neighborhood of $(a,b)$	may be a knight. We make use of this fact extensively in the following argument. 	
	
		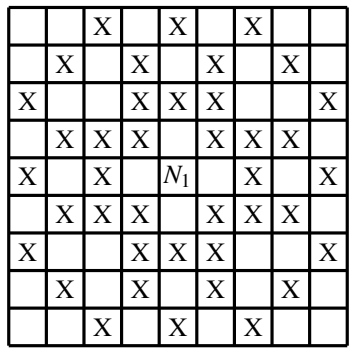
\begin{figure}[H]
		\begin{tikzpicture}
		\foreach \i in {0,...,9} {
			\draw [very thick,black] (.5*\i-2.25,-2.25) -- (.5*\i-2.25,2.25);
		}
		
		\foreach \i in {0,...,9} {
			\draw [very thick,black] (-2.25,.5*\i-2.25) -- (2.25,.5*\i-2.25);
		}
		
		\draw node at (0,0) {$N_1$};
		\draw node at (0,1) {X};
		\draw node at (0,2) {X};
		\draw node at (0,-2) {X};
		\draw node at (0,-1) {X};
		\draw node at (-2,1) {X};
		\draw node at (-2,0) {X};
		\draw node at (-2,-1) {X};
		\draw node at (-1.5,1.5) {X};
		\draw node at (-1.5,0.5) {X};
		\draw node at (-1.5,-0.5) {X};
		\draw node at (-1.5,-1.5) {X};
		\draw node at (-1,2) {X};
		\draw node at (-1,.5) {X};
		\draw node at (-1,0) {X};
		\draw node at (-1,-.5) {X};
		\draw node at (-1,-2) {X};
		\draw node at (-.5,1.5) {X};
		\draw node at (-.5,1) {X};
		\draw node at (-.5,.5) {X};
		\draw node at (-.5,-1.5) {X};
		\draw node at (-.5,-1) {X};
		\draw node at (-.5,-.5) {X};
		\draw node at (2,1) {X};
		\draw node at (2,0) {X};
		\draw node at (2,-1) {X};
		\draw node at (1.5,1.5) {X};
		\draw node at (1.5,0.5) {X};
		\draw node at (1.5,-0.5) {X};
		\draw node at (1.5,-1.5) {X};
		\draw node at (1,2) {X};
		\draw node at (1,.5) {X};
		\draw node at (1,0) {X};
		\draw node at (1,-.5) {X};
		\draw node at (1,-2) {X};
		\draw node at (.5,1.5) {X};
		\draw node at (.5,1) {X};
		\draw node at (.5,.5) {X};
		\draw node at (.5,-1.5) {X};
		\draw node at (.5,-1) {X};
		\draw node at (.5,-.5) {X};		
		\end{tikzpicture}
		\caption{Squares which cannot contain a knight so that $N_1$ remains isolated}
		\label{N_1}
	\end{figure}

Suppose that a second knight $N_2$ is in either $(2,2), (2,-2), (-2,2),$ or $(-2,-2).$ WLOG, assume $N_2$ is in $(2,-2).$ Now we must still dominate $(0,1),$ and must do so with a knight in either $(2,2), (0,1), $ or $(-2,2).$ \\
Case 1. 	$N_3 = (2,2).$ 

Note that this forces knights $N_4 = (1,0)$ and $N_5 = (3,0).$ We claim that $(2,0)$ cannot be dominated. Suppose $N_6$ dominates $(2,0),$ then from Figure \ref{N_1}, $N_6 \in \{ (0,1),(3,2),(4,1),(4,-1),(3,-2),(0,-1)  \}.$ But note that $N_4$ dominates $(-1,1),$ $(3,1),$ and $(3,-1)$ so that $N_6 \neq (0,-1),$ $N_6 \neq (4,-1),$ and $N_6 \neq (4,1)$ respectively. Moreover, $N_5$ dominates $(1,-1)$ and $(1,1)$ so that $N_6 \neq (0,1), N_6 \neq (3,-2),$ and $N_6 \neq (3,2)$ respectively.  This implies $(2,0)$ cannot be dominated. Thus $N_3 \neq (2,2).$  \\
Case 2. $N_3 = (0,1).$ 
	
Consider choosing $N_4$ and $N_5$ to dominate $(1,1)$ and $(-1,1)$ respectively. Then we must have that $N_4 \in \{ (0,3),(2,3),(3,2),(3,0),(0,-1),(-1,0)  \}$ and also that $N_5 \in \{ (0,3),(1,0),(0,-1),(-3,0),(-3,2),(-2,3)  \}.$ Now note that $N_3$ dominates $(1,-1)$ so that $(3,0)$ cannot be a knight. Further, $N_2$ dominates $(3,0)$ so that $(2,2)$ cannot be a knight. But then $N_3$ dominates $(2,2)$ so that $(0,3)$ and $(1,0)$ cannot be knights. Moreover, we have that $N_3$ dominates $(1,3), (2,0), (-1,0)$ which implies that $N_4 \neq (3,2), N_4 \neq (0,-1),$ and $N_4 \neq (-1,0).$ Thus $N_4 = (2,3).$ On the other hand, $N_3$ dominates $(2,0), (-1,-1),$ and $(-1,3)$ so that $N_5 \neq (0,-1), N_5 \neq (-3,0)$ and $N_5 \neq (-3,2).$ Thus $N_5 = (-2,3).$ But note that $N_4$ and $N_5$ dominate $(0,2),$ so $(0,2)$ must be a knight. However, this is a contradiction for $(0,2)$ cannot be a knight. Thus, $N_3 \neq (0,1).$ \\
Case 3. $N_3 = (-2,2).$ 
	
Consider dominating $(1,-1)$ with a knight $N_4.$ It must be the case that $N_4 \in \{ (0,1),(3,0),(3,-2),(2,-3),(0,-3),(-1,0)  \},$ but by symmetry we only need to consider $N_4 \in \{ (0,1),(3,0),(3,-2)  \}.$ Moreover, from Case 2, we know that having $(0,0), (2,-2),$ and $(0,1)$ as knights leads to a contradiction. So we consider just $N_4 \in \{(3,0), (3.-2)\}. $\\
Case 3a. $N_4 = (3,0).$

Now consider dominating $(2,0)$ with a knight $N_5.$ Now this implies that $N_5 \in \{ (0,1),(0,-1),(3,2),(3,-2),(4,1),(4,-1)  \}.$ But $N_4$ dominates $(1,1)$ so $N_5 \neq (0,-1)$ and $N_5 \neq (3,2).$ Also $N_4$ dominates $(1,-1)$ so $N_5 \neq (0,1)$ and $N_5 \neq (3,-2).$ Moreover, $N_3$ dominates $(0,1),$ so $(2,2)$ cannot be a knight. But $(2,2)$ is dominated by $N_4$ so $(4,1)$ cannot be a knight. Thus we must have $N_5 = (4,-1).$ 

But now in order to dominate $(4,0)$ with a knight $N_6,$ we must have $N_6 \in \{ (3,2),(3,-2),(6,1),(6,-1) \}.$ But $N_4$ dominates $(4,2)$ and $(4,-2)$ so $N_6 \neq (6,1)$ and $N_6 \neq (6,-1)$ respectively. Moreover, $N_5$ dominates $(3,1)$ and $(3,-3)$ so that $N_6 \neq (5,2)$ and $N_6 \neq (5,-2)$ respectively. Hence $(4,0)$ cannot be dominated. Thus $N_4 \neq (3,0).$ \\
Case 3b. $N_4 = (3,-2).$ 

First note that from our argument in Case 2, $(1,0), (-1,0), (0,1),$ and  $(0,-1)$ cannot be knights. Now consider dominating $(3,-1),$ with a knight $N_5.$ Then $N_5 \in \{(2,-3), (4,1),(4,-3), (5,0),(5,-2) \}.$ Now $N_4$ dominates $(1,-1), (2,0), (2,-4),$ and $(4,0),$ so $N_6 \neq (2,-3),$ $N_6 \neq (4,1),$ $N_6 \neq (4,-3),$ and $N_6 \neq (5,-2)$ respectively. Thus $N_5 = (5,0).$ 

Now consider dominating $(1,1)$ with a knight $N_6.$ Then we must have $N_6 \in  \{ (0,3),(2,3),(3,2),(3,0) \}.$ But note $N_4$ dominates $(2,0)$ and $(1,-1)$ so $N_6 \neq (3,2),$ $N_6 \neq (3,0).$ Moreover, $N_5$ dominates $(4,2),$ so $N_6 \neq (2,3).$ Thus $N_6 = (0,3).$ 

Now consider dominating $(3,-3)$ with a knight $N_7.$ This implies that $N_7 \in \{ (4,-1), (5,-4), (4,-5), (2,-5), (1,-4)  \}.$ Now $N_4$ dominates $(2,0), (2,-4),$ and $(1,-3)$  so $N_7 \neq (4,-1),$ $N_7 \neq (4,-5)$ and $N_7 \neq (2,-5).$ Moreover, $N_5$ dominates $(4,-2)$ so $N_7 \neq (5,-4).$ Thus $N_7  = (1,-4).$ 

Lastly, consider dominating $(-1,-1)$ with a knight $N_8$ and note that we have $N_8 \in \{ (0,-3), (-2,-3), (-3,-2), (-3,0)  \}.$ Now $N_4$ dominates $(2,-4)$ so $N_8 \neq (0,-3).$ Also $N_7$ dominates $(0,-2)$ and $(-1,-3),$ so $N_8 \neq (-2,-3)$ and $N_8 \neq (-3,-2).$ Moreover, $N_6$ dominates $(-1,1)$ so $N_8 \neq (-3,0).$ Therefore, $(-1,-1)$ cannot be dominated. Thus $N_4 \neq (3,-2).$ 

Therefore, there is no $N_4$ which we can choose to dominate $(1,-1).$ Hence $N_3 \neq (-2,2).$ It follows that given $N_1 = (0,0),$ $N_2 = (2,-2),$ we cannot find $N_3$ to dominate $(0,1).$ Therefore, WLOG any nontrivial perfect dominating set with an isolated knight $(0,0),$ cannot contain a knight in $(2,-2), (2,2), (-2,2),$ or $(-2,-2).$

So now consider the case that $N_1 = (0,0)$ and each of $(2,2), (2,-2),$ $(-2,2),$ and $(-2,-2)$ are not knights. Then we must have that $(0,1)$ and $(1,0)$ are dominated by themselves. However, placing knights in $(0,1)$ and $(1,0)$ forces a knight in $(2,2).$ Which is a contradiction. 

Therefore, any perfect dominating set of $KN_{\mathbb{Z},\mathbb{Z}}$ cannot contain an isolated knight.

It follows then that a perfect dominating set of minimum density is attained when each knight is adjacent to exactly one other knight. In such case, any given knight uniquely dominates 8 squares (including itself). Hence the knights populate $1/8$ of all squares.  	
\end{proof}

\begin{prop}\label{infinite bands}
	$KN_{\mathbb{Z},m}$ and $KN_{\mathbb{N},m}$ contain nontrivial perfect dominating sets if and only if $m = 2,3,4.$ Moreover, these perfect dominating sets have minimum densities of $1/3,$ (at most) $5/12,$ and $1/2,$ respectively. 
\end{prop}

\begin{proof}
	That $KN_{\mathbb{N},m}$ has no nontrivial perfect dominating set for $m \geq 5$ follows from our proof Theorem \ref{5x5}. For $2 \leq m \leq 4,$ nontrivial perfect dominating sets for $KN_{\mathbb{N},m}$ can be constructed by extending the patterns in the perfect dominating sets for $KN_{n,2}, KN_{8n,3},$ and $KN_{2n,4}.$ Moreover, using these patterns gives dominating sets of densities $1/3,$ $5/12,$ and $1/2$ respectively. The density of $1/3$ for $KN_{\mathbb{N},2}$ is optimal since the perfect dominating set is an efficient dominating set. The density of $1/2$ for $KN_{\mathbb{N},4}$ is also optimal, as there is a unique (up to symmetry) perfect dominating set for $KN_{\mathbb{N},4}$ per the proof of Proposition \ref{4_rows}. The density of $5/12$ for $KN_{\mathbb{N},3}$ has only been shown to necessitate an upper bound.  
	
	Furthermore, these patterns can actually be extended in either direction to give nontrivial perfect dominating sets of $KN_{\mathbb{Z},2}, KN_{\mathbb{Z},3},$ and $KN_{\mathbb{Z},4}.$ The corresponding optimal densities of knights remain the same. 
	
	To show that $KN_{\mathbb{Z},m}$ has no nontrivial perfect dominating set for $m \geq 5,$ we consider dominating a sub-board on the boundary with 3 rows and $k$ columns. Starting with $k=4,$ we constructed all sets of knights which were necessary to perfectly dominate only this sub-board. We successively increased $k$ by $1$ and extended each of our constructions as necessary to dominate the three additional squares. When $k=12,$ none of the constructions had a nontrivial extension. Thus, since a perfect dominating set of $KN_{\mathbb{Z},m}$ must perfectly dominate this sub-board of 3 rows and 12 columns on the boundary, and any such set necessitates placing knights on the entire board, $KN_{\mathbb{Z},m}$ has no nontrivial perfect dominating set. 	
\end{proof}

\begin{prop}
	There is no nontrivial perfect dominating set for $KN_{\mathbb{N},\mathbb{N}}$ or $KN_{\mathbb{Z}, \mathbb{N}}.$ 
\end{prop}

\begin{proof}
	That $KN_{\mathbb{N},\mathbb{N}}$ has no nontrivial perfect dominating set again follows directly from our proof of Theorem \ref{5x5}. 
	
	That $KN_{\mathbb{Z},\mathbb{N}}$ has no nontrivial perfect dominating set follows directly from our argument in Proposition \ref{infinite bands} that $KN_{\mathbb{Z},m}$ has no nontrivial perfect dominating set for $m \geq 5.$	
\end{proof}

\noindent \textbf{Conclusion.} Exact values or upper bounds of the perfect domination number have been given for all knights graphs except $KN_{8k+1,3}, KN_{8k+2,3},$ and $KN_{8k+3,3}.$

\end{document}